\documentclass[12pt,a4paper]{amsart}
\usepackage{graphics}
\usepackage{epsfig}
\usepackage[all]{xy}
\usepackage{graphicx}
\usepackage{hyperref}
\theoremstyle{plain}
\usepackage{amssymb}
\usepackage[english]{babel}
\usepackage{cite}
\advance\hoffset-20mm \advance\textwidth40mm

\def\Z{\mathbb{Z}}

\newtheorem{thm}{Theorem}
\newtheorem{thm1}{Theorem}
\newtheorem{thm2}{Theorem}
\newtheorem{thm3}{Theorem}

\newtheorem{thm5}{Theorem}
\newtheorem{thm6}{Theorem}
\newtheorem{prop}[thm2]{Proposition}
\newtheorem{lem}[thm3]{Lemma}
\newtheorem{Rem}[thm6]{Remark}
\newtheorem{Cor}[thm1]{Corollary}

\newtheorem{ex}[thm5]{Example}

\newtheorem*{theo*}{Theorem}

\theoremstyle{definition}

\newtheorem*{definition*}{Definition}

\begin{document}
\sloppy
\title[On Lie isomorphisms of rings]
{On Lie isomorphisms of rings}
\author
{Oksana Bezushchak}
\address{Oksana Bezushchak: Faculty of Mechanics and Mathematics,
Taras Shevchenko National University of Kyiv, 60, Volodymyrska street, 01033  Kyiv, Ukraine}
\email{bezushchak@knu.ua}
\author{Iryna Kashuba}
\address{Iryna Kashuba: Shenzhen International Center for Mathematics, Southern University of Science and Technology, 1088 Xueyuan Avenue, Nanshan District, Shenzhen, Guangdong, China}
          \email{kashuba@sustech.edu.cn}       
\author{Efim Zelmanov}
\address{Efim Zelmanov: Shenzhen International Center for Mathematics, Southern University of Science and Technology, 1088 Xueyuan Avenue, Nanshan District, Shenzhen, Guangdong, China}
          \email{efim.zelmanov@gmail.com}

\begin{abstract}
An associative ring $A$ gives rise to the Lie ring 
$A^{(-)}=(A,[a,b ]=ab-ba)$. The subject of isomorphisms of Lie rings $A^{(-)}$ and $[A,A]$ has attracted considerable  attention in the literature. We prove that if the identity element of $A$ decomposes into a sum of at least three  full orthogonal idempotents, then any isomorphism from the Lie ring $[A,A]$ to the Lie ring $[B,B]$ is standard. 

For non-unital rings, the description is more intricate. Under a certain assumption on idempotents, we extend a Lie isomorphism from  $[A,A]$ to $[B,B]$ to a homomorphism of associative rings $\widehat{A\oplus A^{op}}\to B,$ where $A^{op}=(A,a\cdot b= b\cdot a),$ and $\widehat{A\oplus A^{op}}\to A\oplus A^{op}$ is the universal annihilator extension of the ring $A\oplus A^{op}.$

The results obtained are then applied to the description of automorphisms and  derivations of Lie algebras of infinite matrices.
\end{abstract}

\maketitle

\section{Introduction}

An associative ring $A$ gives rise to the Lie ring 
$A^{(-)}=(A,[\,,\,])$, where $[a,b]=ab-ba$, $a,b\in A$. 
I.~N.~Herstein raised the question of the  extent to which the Lie rings $A^{(-)}$ or the Lie ring of skew-symmetric 
elements with respect to an involution determine the structure
of the associative ring $A$. More precisely,
given another associative ring $B$, I.~N.~Herstein in \cite{Her},  formulated a series of conjectures on Lie isomorphisms 
(i.e., isomorphisms of Lie rings) 
$[A,A]\to [B,B]$ and similar isomorphisms of Lie rings of skew-symmetric elements.

 If $\varphi: A\to B$ is a 
homomorphism of associative rings, then $\varphi$ is also a homomorphism of the Lie rings 
$A^{(-)}\to B^{(-)}$ (i.e., Lie homomorphism). If $\varphi: A\to B$ is an 
anti-homomorphism, i.e., $\varphi(ab)=\varphi(b)\varphi(a)$ for arbitrary elements $a,b\in A$, then  
$-\varphi$ is a homomorphism of the Lie rings 
$A^{(-)}\to B^{(-)}$. 

We call an additive mapping $\varphi: A\to B$ {\it standard} if there exists 
a homomorphism $\psi_1:A\to B$ and an anti-homomorphism $\psi_2: A\to B$ such that: 
\begin{enumerate}
	\item[(i)]  $\psi_1(A) \psi_2(A)=\psi_2(A)\psi_1(A)=(0)$;
	\item[(ii)] $\varphi=\psi_1-\psi_2$.
\end{enumerate}
\noindent Clearly,  a standard mapping $\varphi$ is a homomorphism of the Lie rings $A^{(-)}\to B^{(-)}$.

Let $A^{op}$ be the opposite ring of $A$, i.e., it has the same additive structure, but $a^{op} b^{op}=(ba)^{op}$ for $a$, $b\in A$. Consider the direct sum $A\oplus A^{op}$. The mapping $a\to a\oplus (-a^{op})$, $a\in A$, is an embedding of Lie rings $A^{(-)}\hookrightarrow (A\oplus A^{op})^{(-)}$. It is straightforward to check that an additive mapping $\varphi: A\to B$ is standard if and only if there exists a homomorphism $\chi: A\oplus A^{op}\to B$ of associative rings such that the diagram 

$$
\xymatrix{A\ar[rr]\ar[rrd]_{\varphi}& &
	A\oplus A^{op}\ar[d]^{\chi}
	\\ &&B} \qquad \quad \xymatrix{a\ar[rr]\ar[rrd]& &
	a\oplus(-a^{op})\ar[d]^{\chi}
	\\ &&\varphi(a)}$$
is commutative.

W.~S.~Martindale \cite{Mar} showed that for arbitrary prime unital rings $A$ and $B$ without $2$ and $3$ torsion, 
an isomorphism of Lie rings $A^{(-)}\to B^{(-)}$ is a standard mapping, 
provided that $A$ contains a non-trivial idempotent. In the breakthrough paper \cite{Bresar}, M.~Bre\v sar removed the idempotent assumption. In a subsequent series of papers \cite{BBCM,BBCM2,BBCM3}, K.~Beidar, M.~Bre\v sar, M.~Chebotar,
and W.~S.~Martindale completely proved Herstein's conjectures for prime rings without $2$-torsion.

In this paper, we drop the assumptions on (semi)primeness and additive torsion. Instead, we assume
the existence of sufficiently many pairwise orthogonal full idempotents. 
Recall that an idempotent 
$e\in A$ is called  a {\it full idempotent} if $AeA=A$.
\begin{thm}\label{main}
	Let $A$ be an associative unital ring containing three pairwise orthogonal full idempotents $e_1$, $e_2$, $e_3$ with 
	$e_1+e_2+e_3=1$. Let $B$ be another associative ring. Then an 
	arbitrary Lie isomorphism $[A,A]\to[B,B]$ extends to a standard one.
\end{thm}

The assumptions of Theorem~\ref{main} are satisfied for matrix rings $A=M_n(R)$, $n\geq 3$, over a unital associative ring $R$. In Section~\ref{section5}, we construct an example showing that Theorem~\ref{main} may not hold for rings containing only two orthogonal full idempotents $e_1$, $e_2$ with $e_1+e_2=1$. 

In \cite{Chebotar}, M.~Chebotar  proved that Lie isomorphisms of matrix rings $M_n(\mathbb{F})$, $n\geq 3$, over an algebraically closed field of characteristic $2$ are standard. With applications to Lie rings of infinite matrices in mind, we also consider Lie isomorphisms of non-unital rings. In this context, the description becomes more intricate.

A Lie ring $L$ is called {\it perfect} if $L=[L,L]$. Let $L$ be a perfect Lie ring. An epimorphism of Lie rings $\theta: L'\to L$
is called a {\it central extension} if the Lie ring $L'$ is also perfect and $\operatorname{Ker}\varphi$ lies in the center of  $L'$. I.~Schur \cite{Schur} (for groups) and H.~Garland \cite{Garland} (for Lie algebras) showed that there exists a unique {\it universal central extension} $u:\widehat{L}\to L$, such that for any central extension $\theta: L'\to L$, there exists a homomorphism $\chi:\widehat{L}\to L'$ making the diagram 
$$
\xymatrix{\widehat{L}\ar[rr]^{\chi}\ar[rrd]_u& &
	L'\ar[d]^{\theta}
	\\ &&L}
$$
commutative.

Let $A$ be an associative ring such that $A=AA$. The subset $\operatorname{Ann}(A)=\{a\in A\,|\, aA=Aa=(0)\}$
is called the {\it annihilator} of $A$. An epimorphism of associative rings $\psi: A'\to A$ is called an {\it annihilator extension}
if $A'=A'A'$ and $\operatorname{Ker}\psi\subseteq \operatorname{Ann} (A')$. In Section~\ref{Lie-section}, 
we show that there exists a unique universal annihilator extension $\omega: \widehat{A}\to A$, such that for any annihilator extension $\psi: A'\to A,$ there exists a homomorphism $\chi: \widehat{A}\to A'$ making the diagram
$$
\xymatrix{\widehat{A}\ar[rr]^{\chi}\ar[rrd]_{\omega}& &
	A'\ar[d]^{\psi}
	\\ &&A}
$$
commutative. If the Lie ring $[A,A]$ is perfect, then the Lie subring $[[\widehat{A},\widehat{A}],[\widehat{A},\widehat{A}]]$
of $\widehat{A}^{(-)}$ is also perfect. The restriction $w:[[\widehat{A},\widehat{A}],[\widehat{A},\widehat{A}]]\to[A,A]$ is a central extension. Hence, there exists a homomorphism 
\begin{equation}\label{gamma}
	\gamma: \widehat{[A,A]}\to [[\widehat{A},\widehat{A}],[\widehat{A},\widehat{A}]] 
\end{equation}
making the diagram 
$$
\xymatrix{ \widehat{[A,A]}\ar[rr]^{\gamma}\ar[rrd]_u& &
	\widehat{A} \ar[d]^{\omega}
	\\ && A}
$$
commutative.

We now state the following theorem.
\begin{thm}\label{main2}
	Let $A$ be an associative ring containing three pairwise orthogonal idempotents $e_1$, $e_2$, $e_3$ such that 
	$Ae_i A=A e_4 A=A$, $1\leq i\leq 3$, where $e_4=1-e_1-e_2-e_3$.  Let $B$ be another associative ring, and let $\varphi:[A,A]\to[B,B]$
	be a Lie isomorphism. Then there exists a unique homomorphism $\chi: \widehat{A\oplus A^{op}}\to B$ making the diagram
	$$
	\xymatrix{ \widehat{[A,A]}\ar[rr]\ar[rrd]_{\varphi\circ u=\widehat{\varphi}}& &
		\widehat{A\oplus A^{op}}\ar[d]^{\chi}
		\\ &&B}
	$$
	commutative.
\end{thm}
\begin{Rem} We do not assume that the ring $A$ contains $1$. The idempotent $e_4$ lies in the unital hull $A+{\mathbb Z}\cdot 1$. 
\end{Rem}

 Let $\langle[B,B]\rangle$ be the subring of  $B$ generated by $[B,B]$.
\begin{Cor}\label{corollary-main2} If $\operatorname{Ann}(\langle[B,B]\rangle)=(0)$, then any  Lie isomorphism 
	$[A,A]\to[B,B]$ extends to a standard homomorphism.
\end{Cor}
\begin{thm}\label{main3}
	Let $A$ be an associative ring either satisfying the assumptions of Theorem~\ref{main} or  
	Theorem~\ref{main2} with  $\operatorname{Ann}(A)=(0)$. Then any derivation of the Lie ring 
	$[A,A]$ extends to a derivation of the associative ring $A$.
\end{thm}
In \cite{Bezushchak}, we studied Lie automorphisms and Lie derivations of locally matrix algebras of characteristic $\neq 2$; see \cite{BezOl_2,Bez_Ol__Sush}. Automorphisms and derivations of classical Lie algebras of infinite matrices over fields of characteristic $\neq 2$ were 
described in \cite{Bezushchak2}. Using Theorems~\ref{main} and \ref{main2}, we can drop the characteristic assumption.

Let us recall the necessary definitions. Let $I$ be an infinite set, and let $\mathbb{F}$ be a field. The algebra $M_{\infty}(I,\mathbb{F})$
consists of $I\times I$ matrices over $\mathbb{F}$ with only many nonzero entries.  We also consider  the Lie algebra: 
$$
\mathfrak{sl}_{\infty}(I,\mathbb{F})=[M_{\infty}(I,\mathbb{F}),M_{\infty}(I,\mathbb{F})].
$$ Let $M_{rcf}(I,\mathbb{F})$ denote the algebra of $I\times I$ matrices having finitely many nonzero entries 
in each row and column. 
Let $GL_{rcf}(I,\mathbb{F})$ denote the group of invertible elements of  $M_{rcf}(I,\mathbb{F})$. Theorem~\ref{main2} implies:
\begin{Cor}\label{main4}$[$see Theorem $6$ in \cite{Bezushchak2}$]$
	Let $\mathbb{F}$ be an arbitrary field. Any automorphism of the Lie algebra $L=\mathfrak{sl}_{\infty}(I,\mathbb{F})$ 
	is of the form
	$$\varphi(a)=x^{-1} a x, \ a\in L, \quad \text{or} \quad \varphi(a)=-x^{-1}a^t x, \ a\in L,$$ where $x\in GL_{rcf}(I,\mathbb{F})$
	and $a^t$ is the transpose of $a$.
\end{Cor}
Regarding derivations, the most general result is:
\begin{Cor}\label{main5}$[$see Theorem $1$ in \cite{Bezushchak_Res_Math}$]$
	Let $R$ be an arbitrary associative ring with $1$. Any derivation of the Lie ring 
	$\mathfrak{sl}_{\infty}(I,R)$  is of the form $d=ad(a)+u$, where
	$a\in M_{rcf}(I,R)$, and $u$ is a derivation of the ring $R$.
\end{Cor}

\section{Root Graded Lie Rings}\label{Lie-section}
The following definition is inspired by the Tits-Kantor-Koecher construction \cite{Kantor,Koecher,Tits}, see also \cite{Z}. Consider an integer lattice $$\Gamma=\bigoplus_{i=1}^n \mathbb{Z}\, \omega_i$$ and a root system $\Delta\subset \Gamma$.
A Lie ring $L$ is said to be {\it $\Delta$-graded} if $L$ is $\Gamma$-graded and 
\begin{equation}\label{delta-graded}
L=L_0+\sum_{\alpha\in\Delta} L_{\alpha}, \quad L_0=\sum_{\alpha\in\Delta}[L_{-\alpha},L_{\alpha}].
\end{equation}
\begin{Rem}
	This definition of a $\Delta$-graded Lie ring  differs from the corresponding definitions in \cite{BZ} and \cite{BM}.
\end{Rem}
 Let $A$ be an associative ring, and let $A+\mathbb{Z} \cdot 1$ be the unital hull of $A$ (if  $A$ is not unital).
	Let $e_1,\dots,e_n$ be pairwise orthogonal full idempotents in $A$ or in $A+ \mathbb{Z}\cdot 1$ if $A$ is not unital, such that $\sum_{i=1}^n e_i=1$ with $n\geq 3$. The ring $A$ is graded by the integer lattice $$\Gamma=\bigoplus_{i=1}^n\Z\, \omega_i, \quad \text{where} \quad A_0=\sum_{i=1}^n e_i Ae_i, \quad A_{\omega_i-\omega_j}=e_i Ae_j, \quad
 1\leq i\neq j\leq n.$$ Let $\Delta=\{\omega_i-\omega_j\,|\, 1\leq i\neq j\leq n\}$ be the root system of type $A_{n-1}$. 
 Then 
 \begin{equation*}\label{delta-gradingA}
 A=A_0+\sum_{1\leq i\neq j\leq n} e_i A e_j= A_0+\sum_{\alpha\in\Delta} A_{\alpha}.
 \end{equation*}
 \begin{lem}\label{example-graded}
 The Lie ring $[A,A]$ is $\Delta$-graded:
 \begin{equation*}\label{delta-grading[A,A]}
 [A,A]=\sum_{1\leq i\neq j\leq n}[e_i A e_j,e_j A e_i]+\sum_{1\leq i\neq j\leq n} e_i Ae_j.
 \end{equation*}
 \end{lem}
 \begin{proof} The Lie ring 
 $$L'=\sum_{1\leq i\neq j\leq n}[e_i Ae_j, e_j Ae_i]+\sum_{1\leq i\neq j\leq n} e_i Ae_j$$ is  $\Delta$-graded. Obviously $L'\subseteq [A,A]$.
 On the other hand, 
 $$
 [A,A]=[\sum_{1\leq k,s\leq n} e_k A e_s, \sum_{1\leq p,q\leq n} e_p A e_q]=\sum_{1\leq k,s,p,q\leq n} [e_k A e_s, e_p A e_q].
 $$
If $s\neq p$ or $k\neq q$, then $[e_k A e_s, e_p A e_q]\subseteq \sum_{1\leq i\neq j\leq n} e_i A e_j$. Thus, to guaranty that $[A,A]\subseteq L'$, 
it suffices to show that  $$[e_i A e_i, e_i A e_i]\subseteq \sum_{j=1,j\neq i}^n[e_iAe_j,e_j A e_i].$$ 
	Let $a,b\in A$. Choose $j\neq i$. Then $$e_i a e_i=\sum_k e_i x^{(k)}e_j y^{(k)}e_i \quad \text{for some elements} \quad x^{(k)}, y^{(k)}\in A.$$
	We have $$[e_i x^{(k)}e_j, e_j y^{(k)}e_i]=e_i x^{(k)}e_j y^{(k)}e_i-e_j y^{(k)} e_i x^{(k)}e_j ,$$ hence 
	$$
	[e_i ae_i, e_i b e_i]=\sum_{k}[[e_i x^{(k)}e_j, e_j y^{(k)}e_i], e_i be_i]\in[e_i A e_j, e_jA e_i].
	$$
This completes the proof of the lemma.
\end{proof}
\begin{Cor}\label{Corollary_4+} The Lie ring $[A,A]$ is perfect.
\end{Cor}
\begin{proof} Indeed, assuming $i,j,k$ are distinct, $e_i Ae_j=[e_i A e_k,e_k Ae_j]\subseteq [[A,A],[A,A]]$, and $[e_i Ae_j, e_j A e_i]\subseteq  [[A,A],[A,A]]$.
\end{proof}
Let $L$, $L'$ be $\Delta$-graded Lie rings. A surjective $\Gamma$-graded homomorphism $\varphi: L'\to L$ is called a {\it graded central extension} if $\operatorname{Ker}\varphi\subseteq L_0'$. Since $L_0'=\sum_{\alpha\in \Delta}[L'_{\alpha},L'_{-\alpha} ]$, it follows that $\operatorname{Ker}\varphi$ lies in the center of $L'$. 

A graded central extension $u: \widehat{L}_{gr}\to L$ is called {\it universal} if, for any graded central 
extension $\varphi: L'\to L$, there exists a $\Gamma$-graded homomorphism $\chi: \widehat{L}_{gr}\to L'$ such that 
the diagram
$$
\xymatrix{\widehat{L}_{gr}\ar[rr]^{\chi}\ar[rrd]_u& &
	L'\ar[d]^{\varphi}
	\\ &&L}
$$
is commutative.
\begin{lem}\label{extension}
	There exists a unique universal graded central extension of a $\Delta$-graded Lie ring. 
\end{lem}
\begin{proof}
	Let $E_{\alpha}$ be a generating system of the abelian group $L_{\alpha}$, $\alpha\in\Delta$. Let $X_{\alpha}$
	be a copy of the set $E_{\alpha}$, and let $\psi_{\alpha}: X_{\alpha}\to E_{\alpha}$ be a bijection. Define  $X=\bigsqcup _{\alpha\in\Delta} X_{\alpha}$, the disjoint union of the  $X_{\alpha}$'s. Let $Lie\langle X\rangle$ be the free Lie ring on the set of free generators
	$X$, and let $\psi: X\to \cup_{\alpha\in\Delta}E_{\alpha}$ be the mapping that extends the mappings $\psi_{\alpha}$,
	$\alpha\in\Delta$. The mapping $\psi$ extends to a homomorphism  $\psi: Lie\langle X\rangle \to L$. Since the mapping $\psi: X\to \cup_{\alpha\in\Delta} E_{\alpha}$ is $\Gamma$-graded, it follows that the homomorphism $\psi$
	is $\Gamma$-graded as well. Then $I=\operatorname{Ker}\psi=\Sigma_{\alpha\in\Gamma} I_{\alpha}$. Consider the Lie ring $\widehat{L}_{gr}$, which is presented by generators $X$ and the set of relations $\Sigma_{0\neq \alpha\in\Gamma} I_{\alpha}$. Since $\Sigma_{0\neq \alpha\in\Gamma} I_{\alpha}\subset I$, there exists a natural homomorphism $u: \widehat{L}_{gr}\to L$. If $a\in (\widehat{L}_{gr})_{\alpha}$, $\alpha\neq 0$,
	and $u(a)=0$, then $a\in I_{\alpha}$, and therefore $a=0$. We have shown that $\operatorname{Ker}u\subseteq (\widehat{L}_{gr})_0$, thus
	$u: \widehat{L}_{gr}\to L$ is a graded central extension. The universality and uniqueness of $u$ immediately follow from the construction. This completes the proof of the lemma.
\end{proof}
\begin{Rem}
The homomorphism $\widehat{L}_{gr}\to L$ is not a universal central extension of $L$ in the sense of I. Schur \cite{Schur} and H. Garland \cite{Garland}. In general, we do not even assume that the ring $L$ is perfect. 
\end{Rem}
There is an analog of the Schur-Garland construction for annihilator extension of associative rings.
\begin{lem}
	Let $A$ be an associative ring with $A=AA$. There exists a unique universal annihilator extension $\omega: U\to A$, $U^2=U$,
	and $\operatorname{Ker}\omega\subset \operatorname{Ann}(U)$, such that for any annihilator extension $\varphi: A'\to A$,
	there exists a homomorphism $\chi: U\to A'$ making the diagram 
	$$
	\xymatrix{U\ar[rr]^{\chi}\ar[rrd]_{\omega}& &
		A'\ar[d]^{\varphi}
		\\ &&A}
	$$
	commutative.
\end{lem}
\begin{proof}
	We follow the classical scheme from \cite{Garland,Schur}. Let $Ass\langle X\rangle$ be a free associative ring on the set of free 
	generators $X$. Note that $Ass\langle X\rangle$ is a free associative ring without a unit. Let $\mu: Ass\langle X\rangle \to A$
	be an epimorphism, and let $I=\operatorname{Ker}\mu$. We leave it to the reader to verify that the ring 
	$$
	U=Ass\langle X\rangle^2/(I\, Ass\langle X\rangle+Ass\langle X\rangle I)
	$$
	with the natural epimorphism $\omega: U\to A$ is a universal annihilator extension and  is unique. This completes the sketch of the proof.
\end{proof}

Let $L=L_0+\sum_{\alpha\in\Delta} L_{\alpha}$ be a $\Delta$-graded Lie ring (see \eqref{delta-graded}), and let $A$ be an associative ring. Let 
$u:\widehat{L}_{gr}\to L$ be the universal graded central extension of $L$.

We call a Lie homomorphism $\varphi: \widehat{L}_{gr}\to A^{(-)}$  a {\it  specialization} if 
$\varphi(L_{\alpha_1})\cdots\varphi(L_{\alpha_n})=(0)$ for any $\alpha_1,\dots,\alpha_n\in \Delta$ such that $\alpha_1+\dots+\alpha_n\notin \Delta\cup\{0\}$.
Since the root systems that arise in Theorems \ref{main} and \ref{main2}  are $A_2$ and $A_3$,
respectively, we  assume that
$$
\Delta=\{\omega_i -\omega_j\,|\,1\leq i\neq j\leq n\}\subset \bigoplus_{i=1}^n \mathbb{Z}\, \omega_i, \quad n\geq 3.
$$
\begin{lem}\label{tech-lemma} 
	Let $L=L_0+\Sigma_{\alpha\in\Delta} L_{\alpha}$ be a $\Delta$-graded Lie ring, let $A$ be an associative ring, and let 
	$\varphi:\widehat{L}_{gr}\to A^{(-)}$ be a Lie homomorphism. If
	$\varphi(L_{\alpha})\varphi(L_{\beta})=0$ for any $\alpha$, $\beta\in\Delta$ such that $\alpha+\beta\notin\Delta\cup\{0\}$, then $\varphi$ is a specialization.
\end{lem}
\begin{proof}
	An arbitrary mapping $g\,:\,\{\omega_1,\dots,\omega_n\}\to\{-1,1\}$ gives rise to a homomorphism 
	$$f_g:\bigoplus_{i=1}^n \mathbb{Z}\, \omega_i\to\mathbb{Z}, \qquad f_g(\omega_i)=g(\omega_i), \quad 1\leq i\leq n.$$
	All roots $\alpha=\omega_i-\omega_j$, $1\leq i\neq j\leq n$, have values  $-2,$ $0$ or $2$.
	If $\gamma\notin \Delta\cup\{0\}$, then there exists $g$ such that $f_g(\gamma)>2$. We need to show that 
	$$\varphi(L_{\alpha_1})\cdots\varphi(L_{\alpha_m})=(0), \quad m\geq 2,$$ whenever $\alpha_1,\dots,\alpha_m\in\Delta$ and $\gamma=\alpha_1+\dots+\alpha_m\notin \Delta\cup\{0\}$. By the assumption of the lemma, the statement is true for $m=2$.
	Let $m\geq 3$ and assume that the assertion holds for products of length $<m$. Note that we can permute any two factors in $\varphi(L_{\alpha_1})\varphi(L_{\alpha_2})\cdots\varphi(L_{\alpha_m})$.   Indeed, if 
	$\alpha_{i}+\alpha_{i+1}\notin \Delta\cup\{0\}$, $1\leq i<m$, then $\varphi(L_{\alpha_i})\varphi(L_{\alpha_{i+1}})=(0)$ by the assumption. Suppose that $\alpha_{i}+\alpha_{i+1}\in \Delta\cup\{0\}$. We have $$
	\begin{array}{l} \varphi(L_{\alpha_1}) \cdots \varphi(L_{\alpha_i})\varphi(L_{\alpha_{i+1}})\cdots\varphi(L_{\alpha_m})\subseteq\varphi(L_{\alpha_1})\cdots \varphi(L_{\alpha_{i+1}})\varphi(L_{\alpha_i})\cdots\varphi(L_{\alpha_m})+\vspace{0.3cm}\\
    \qquad\qquad\qquad +
    \varphi(L_{\alpha_1})\cdots \varphi([L_{\alpha_i},L_{\alpha_{i+1}}])\cdots\varphi(L_{\alpha_m}).
		\end{array}
	$$
	The product $\varphi(L_{\alpha_1})\cdots \varphi([L_{\alpha_i},L_{\alpha_{i+1}}])\cdots\varphi(L_{\alpha_m})$ has length $m-1,$ and  thus, by the induction hypothesis on $m$, it is equal to zero. Consequently, for an arbitrary permutation $\pi$ on $m$ symbols,  we have $$\varphi(L_{\alpha_1})\cdots\varphi(L_{\alpha_m})=\varphi(L_{\alpha_{\pi(1)}})\cdots\varphi(L_{\alpha_{\pi(m)}}).$$  
    Suppose that $\gamma = \alpha_{1}+\cdots +\alpha_{m}\notin \Delta\cup\{0\}$. Then there exists $g:\{w_1,\ldots, w_n\}\to \{-1,1\}$ such that $f_g (\gamma )> 2.$ Since $f_g (\gamma )=f_g(\alpha_1)+\cdots + f_g(\alpha_m)$, it follows that $f_g(\gamma)\geq 4.$ Then there exist $1\leq i< j\leq n$ such that
	$f_g(\alpha_i)=f_g(\alpha_j)=2$. By moving these two factors together, we obtain
	$$
	\varphi(L_{\alpha_1})\cdots\varphi(L_{\alpha_m}) =\varphi(L_{\alpha_i})\varphi(L_{\alpha_j})\cdots=(0),
	$$
	since $\alpha_i+\alpha_j\notin\Delta\cup\{0\}$. This completes the proof of the lemma.
\end{proof}
Let $U$ be an associative ring. A specialization $u: \widehat{L}_{gr}\to U$ is called \textit{universal} if $U$ is generated by the subset $\cup_{\alpha\in\Delta} u(L_{\alpha})$, and,   for 
	an arbitrary specialization $\varphi :\widehat{L}_{gr}\to A$ in an associative ring $A$, there exists a homomorphism 
	$\chi\,:\, U\to A$ of associative rings such that the diagram
	$$
	\xymatrix{\widehat{L}_{gr}\ar[rr]^u\ar[rrd]_{\varphi}& &
		U\ar[d]^{\chi}
		\\ &&A}
	$$
	is commutative.
\begin{lem}
	For an arbitrary $\Delta$-graded Lie ring $L=L_0+\Sigma_{\alpha\in\Delta} L_{\alpha}$, there exists a unique 
	universal specialization.  
	\end{lem}
\begin{proof}
	Recall the construction of Lemma~\ref{extension}. Choose a set of generators $E_{\alpha}$
	in the abelian group $L_{\alpha}$. Let $X_{\alpha}$ be a copy of the set $E_{\alpha}$, $\psi_{\alpha}: X_{\alpha}\to E_{\alpha}$  a bijection, and $X=\bigsqcup_{\alpha\in\Delta} X_{\alpha}$. Let $\psi: Lie\langle X\rangle\to L$ be the natural $\Gamma$-graded epimorphism
	 extending the mapping  $\cup_{\alpha\in\Delta}\psi_{\alpha}: X \to L$. Let  $I=\operatorname{Ker}\psi=\sum_{\alpha\in\Gamma}I_{\alpha}$, and define $$\widehat{L}_{gr}=Lie\big\langle X\,|\, \sum_{0\neq \alpha\in\Gamma} I_{\alpha}=(0)\big\rangle.$$ Consider the free associative ring $A\langle X \rangle$
	on the set of free generators $X$, where  $Lie\langle X\rangle\subset A\langle X\rangle^{(-)}$.

Let $J$ be the ideal of $A\langle X \rangle$ generated by $\sum_{0\neq \alpha\in\Gamma} I_{\alpha}$ and by the set $\cup X_{\alpha_1}\dots X_{\alpha_r}$, with
	$\alpha_1+\dots +\alpha_r\notin\Delta\cup\{0\}$.

Let $U=A\langle X\rangle/J$. Since $\sum_{0\neq \alpha\in\Gamma} I_{\alpha}\subset J$, we obtain a natural Lie ring homomorphism $u:\widehat{L}\to U^{(-)}$. Since elements $X_{\alpha_1}\dots X_{\alpha_r}$, where $\alpha_1+\dots +\alpha_r\notin\Delta\cup\{0\}$,  belong to $J$, the mapping $u$ is a specialization. 
 The universality and uniqueness of the specialization follow immediately. This completes the proof of the lemma.
\end{proof}

From the construction of the ideal $J$, it follows that the ring $U$ is $\Gamma$-graded and  can be written as $$U=U_0+\Sigma_{\alpha\in\Delta} U_{\alpha}.$$
The free associative ring $A\langle X\rangle$ is equipped with a unique involution $*: A\langle X\rangle\to A\langle X\rangle$, such that $x^*=-x$ for any generator $x\in X$. Consequently, every element of $Lie\langle X\rangle$ is
skew-symmetric with respect to $*$. Since $J^*=J$,  the involution $*$ gives rise to an involution on the ring
$U=A\langle X\rangle/J$. We denote this involution by $*$ as well. Moreover, any element from $u(L_{\alpha})$, $\alpha\in\Delta$,
is skew-symmetric.

\section{A Lie Isomorphism is a Specialization}
Let $A$ be an associative ring satisfying the assumptions of Theorem \ref{main} and Theorem \ref{main2}. In Lemma~\ref{example-graded}, we showed that the Lie ring $[A,A]$ is $\Delta$-graded, where $\Delta=A_2$ in the case of Theorem \ref{main},  or 
$\Delta=A_3$ in the  case of Theorem \ref{main2}. Specifically, $$[A,A]_{\omega_i-\omega_j}=e_i Ae_j=A_{ij}, \quad 1\leq i\neq j\leq 3 \text{\ or\ } 4.$$ For any element $a\in A$, the components $a_{ij}=e_i a e_j\in A_{ij}$ are referred to as its \emph{Peirce components}. We will also use the notation $a_{ij}$ for elements
in $A_{ij}$.

Let $B$ be an associative ring and let $\varphi:[A,A]\to[B,B]$ be a Lie isomorphism.
\begin{prop}\label{speciality}
	The Lie isomorphism $\varphi$ is a specialization.
\end{prop}
\begin{lem}\label{product-of-two-in-commutator}
	If $\alpha, \beta\in\Delta$ and $\alpha+\beta\notin \Delta\cup\{0\}$, then 
	$\varphi(A_{\alpha})\varphi(A_{\beta})\subseteq [B,B]$.
\end{lem}
\begin{proof}
	For any $i,j,k,l,m,n$, we have:
	\begin{equation}\label{derivation-rel}
		[\varphi(a_{ij}),\varphi(a_{kl})\varphi(a_{mn})]=[\varphi(a_{ij}),\varphi(a_{kl})]\varphi(a_{mn})
		+\varphi(a_{kl})[\varphi(a_{ij}),\varphi(a_{mn})].
	\end{equation}
	Let $i\neq j$. We aim to show that $\varphi(A_{ij})^2\subseteq [B,B]$. There exists $k$  distinct  from both $i$ and $j$. Thus,  
	$A_{ij}=[A_{ik},A_{kj}]$, and 
	$$\varphi(A_{ij})^2=[\varphi(A_{ik}),\varphi(A_{kj})]\varphi(A_{ij})\stackrel{(\ref{derivation-rel})}{\subseteq}
	[\varphi(A_{ik}), \varphi(A_{kj})\varphi(A_{ij})]\subset [B,B].$$
	Now, let $i,j,k$ be distinct. We  show that $\varphi(A_{ij})\varphi(A_{ik})\subseteq [B,B]$. As before, $A_{ij}=[A_{ik}, A_{kj}]$. Hence:
	$$
	\varphi(A_{ij})\varphi(A_{ik})=[\varphi(A_{ik}),\varphi(A_{kj})]\varphi(A_{ik})\stackrel{(\ref{derivation-rel})}{\subseteq}[\varphi(A_{ik}),\varphi(A_{kj})\varphi(A_{ik})]\subseteq [B,B].
	$$
	Similarly, $\varphi(A_{ji})\varphi(A_{ki})\subseteq [B,B]$. 
	
	It remains to prove that $\varphi(A_{12})\varphi(A_{34})\subseteq [B,B]$ under
	the assumption of Theorem \ref{main2}. Since $A_{12}=[A_{14},A_{42}]$, we have: 
	$$
	\varphi(A_{12})\varphi(A_{34})=[\varphi(A_{14}),\varphi(A_{42})]\varphi(A_{34})\stackrel{(\ref{derivation-rel})}{\subseteq}[\varphi(A_{14}),\varphi(A_{42})\varphi(A_{34})]\subseteq [B,B].
	$$
	This completes the proof of the lemma.
\end{proof}
Let $i\neq j$, $1\leq i\leq 3$. Since $e_i\in A= Ae_jA$, it follows that there exist elements $a_{ij}^{(\mu)}\in A_{ij}$,  and  $b_{ji}^{(\mu)}\in A_{ji}$, such that:
\begin{equation}\label{def-e} e_i=\sum_{\mu} a_{ij}^{(\mu)}b_{ji}^{(\mu)}.\end{equation} Define $h_{ij}\in[A,A]$ as:
\begin{equation*}\label{def-h} h_{ij}=\sum_{\mu} [a_{ij}^{(\mu)},b_{ji}^{(\mu)}]\in e_i+e_j Ae_j.
\end{equation*}
\begin{lem}\label{zero-commutator-implies-zero}
	Let $k$ be distinct from $i$ and $j$, $1\leq j\leq 3$, and let $x_{ij}$ be an element from $A_{ij}$ such that $[x_{ij}, A_{jk}]=(0)$. Then $x_{ij}=0$.
\end{lem}
\begin{proof}
	We have: $[x_{ij}, A_{jk}]=x_{ij}  A_{jk}=x_{ij}A e_k=(0)$. Since $A=Ae_k A$, it follows that $x_{ij}A=(0)$, and in particular, $x_{ij} e_j=x_{ij}=0$. 
	This completes the proof of lemma. 
\end{proof}
\begin{Rem}\label{symmetry}  Similarly, we can prove the following statement: if $k$ is distinct from $i$ and $j$, $1\leq j\leq 3$, $x_{ji}\in A_{ji}$, and $[x_{ji}, A_{kj}]=(0)$, then $x_{ji}=0$.
\end{Rem}
Denote $Diag(A)=\sum_i A_{ii}$.
\begin{lem}\label{commutator-belongs-diagonal} For arbitrary $i\neq j$, $$\varphi^{-1}(\varphi(A_{ij})\varphi(A_{ij}))\subseteq Diag(A).$$
\end{lem}
\begin{proof}
	Let $X=\varphi^{-1}(\varphi(A_{ij})\varphi(A_{ij}))\subseteq [A,A]$. Since an arbitrary element from $\varphi(A_{ij})\varphi(A_{ij})$
	commutes with any element from $\varphi(A_{ik})$, $k\neq i$, and  $\varphi(A_{tj})$, $t\neq j$, we conclude: 
	\begin{equation}\label{condition 1}
		[X,A_{ik}]=[X,A_{tj}]=(0).
	\end{equation}
	At least one of the indices $i,j$ is less than $4$ (in the case of Theorem \ref{main2}). Let $1\leq i\leq 3$. For any element $a_{ij}\in A_{ij}$,
	$k\neq i,j$, we know $[h_{ik},a_{ij}]=a_{ij}$. Thus, for any $x\in X$, we conclude: 
	\begin{equation}\label{condition 2}[h_{ik},x]=2x.
	\end{equation}
	Indeed, if $x\in X$, there exist $a_{ij}, b_{ij}\in A_{ij}$ such that  $x=\varphi^{-1}(\varphi(a_{ij})\varphi(b_{ij}))$. Then, 
	\begin{equation}\label{double-proof}
		\varphi([h_{ik},x])=[\varphi(h_{ik}),\varphi(a_{ij})\varphi(b_{ij})]\stackrel{(\ref{derivation-rel})}{=}\varphi([h_{ik},a_{ij}])\varphi(b_{ij})+\varphi(a_{ij})\varphi([h_{ik},b_{ij}])=
		2\varphi(x).
	\end{equation}
	If $x\in X$, then the equalities \eqref{condition 1} and \eqref{condition 2} hold for every Peirce component of the element $x$.  Let $k\neq i,j$. We have 
	$[h_{ik}, x_{ij}]=x_{ij}$. In view of \eqref{condition 2}, it implies $x_{ij}=0$. Next, consider the Pierce component $x_{iq}$, $q\neq i,j$. By \eqref{condition 1}, $[x_{iq}, A_{qj}]=(0)$. If $1\leq q\leq 3$, then by Lemma \ref{zero-commutator-implies-zero}, it implies $x_{iq}=0$. Let
	$q=4$ (in the case of Theorem \ref{main2}). Then there exists an index $k$ that is distinct  from $i,j$ and $q$. We have $[h_{ik},x_{i4}]=x_{i4}$,
	which, together with \eqref{condition 2}, implies $x_{i4}=0$. Thus, we have proved that all non-diagonal Peirce components $x_{iq}$, $1\leq q\leq 4$, are equal to $0$.
	Similarly, as shown in  Remark~\ref{symmetry}, all non-diagonal Peirce components $x_{pj}$, $1\leq p\leq 4$, are also equal to $0$.
	
	Let $1\leq j\leq 3$. All Peirce components $x_{pq}$, $p\neq q$, $p=i$ or $q=j$, are equal to zero. Consider a Peirce component
	$x_{pq}$, $p\neq q$, $p\neq i$, $q\neq j$. At least one of the indices $p,q$ is different from $4$. Let $p<4$. If $q\neq i$, then by \eqref{condition 1},
	we have $[A_{ip}, x_{pq}]=(0)$. By Lemma \ref{zero-commutator-implies-zero}, it implies $x_{pq}=0$. Let $q=i$. There exists an index $k$ distinct from $i$ and $p$. By \eqref{condition 1}, $[x_{pi}, A_{ik}]=(0)$. Again, by Lemma \ref{zero-commutator-implies-zero}, we obtain $x_{pi}=0$. It remains to consider the case $j=4$. Consider a Peirce component $x_{pq}$, $p\neq i$. If $q< 4$, then  $[x_{pq},A_{q4}]=(0)$ implies $x_{pq}=0$. If $q=4$,
	then $[A_{ip},x_{p4}]=(0)$ implies $x_{p4}=0$. This completes the proof of the lemma.
\end{proof}
\begin{lem}\label{preimage-in-diagonal}
	For arbitrary distinct indices $i,j,k$, we have $$ \varphi^{-1}(\varphi(A_{ij})\varphi(A_{ik}))\subseteq Diag(A).$$
\end{lem}
\begin{proof}
	At first, suppose that $1\leq i,j,k\leq 3$. Without loss of generality,  consider $Y=\varphi^{-1}(\varphi(A_{12})\varphi(A_{13}))$. As in the proof of
	Lemma \ref{commutator-belongs-diagonal},  
	\begin{equation}\label{condition 3}
		[Y,A_{1t}]=(0), \  t\neq 1.
	\end{equation}
	Then 
	$$
	[\varphi(Y),\varphi(A_{23})]=[\varphi(A_{12})\varphi(A_{13}),\varphi(A_{23})]\stackrel{(\ref{derivation-rel})}{\subseteq}[\varphi(A_{12}),\varphi(A_{23})]\varphi(A_{13})
	\subseteq \varphi(A_{13})\varphi(A_{13}).
	$$
	By Lemma \ref{commutator-belongs-diagonal}, $[Y, A_{23}]\subseteq Diag(A)$.  Similarly,  $[Y, A_{32}]\subseteq Diag(A)$. Let $Y_{ij}=e_i Ye_j$
	be the Pierce components of $Y$. From  $[ A_{13}, Y]=(0)$
	and Lemma \ref{zero-commutator-implies-zero}, it follows that $Y_{32}=(0)$. Similarly, $Y_{23}=(0)$. The diagonal condition implies:
	\begin{equation*}\label{condition 4}
		[Y,A_{23}]=[Y,A_{32}]=(0).
	\end{equation*}
	By Lemma \ref{zero-commutator-implies-zero}, $[Y_{12}, A_{23}]=(0)$ implies $Y_{12}=(0)$. Similarly, $[Y_{21}, A_{13}]=(0)$
	implies $Y_{21}=(0)$. Furthermore, $[Y_{13}, A_{32}]=(0)$ implies $Y_{13}=0$, and $[Y_{31}, A_{12}]=(0)$ implies $Y_{31}=(0)$. Additionally,
	$$[Y_{41}, A_{12}]=[Y_{42}, A_{23}]=[Y_{43}, A_{32}]=(0)$$ implies $Y_{41}=Y_{42}=Y_{43}=(0)$, and  $[A_{32}, Y_{24}]=[A_{23}, Y_{34}]=(0)$
	implies $Y_{24}=Y_{34}=(0)$. We have proved that $Y\subseteq Diag(A)+A_{14}$. Let $y\in Y$, $y=d+a_{14}$, where $d\in Diag(A)$, $a_{14}\in A_{14}$. Furthermore,
	$$\varphi([Y, A_{42}])=[\varphi(A_{12})\varphi(A_{13}),\varphi(A_{42})]=(0)$$ implies $[Y,A_{42}]=(0)$. It  then  follows that $a_{14} A_{42}=(0)$. 
	Since $A=Ae_2 A$, we conclude that $a_{14} A=(0)$. Recall that the element $h_{24}$
	lies in $[A_{24}, A_{42}]$, $h_{24}\in e_2+e_4 A e_4$. Hence, $[a_{14}, h_{24}]=0$. On the other hand, for an arbitrary element 
	$x\in\varphi(A_{12}) \varphi(A_{13})$, the arguments in \eqref{double-proof}  imply $[x,\varphi(h_{24})]=x$. This implies $a_{14}=0$, and 
	therefore, $Y\subseteq Diag(A)$. 
	
	Now we  consider the case $4\in\{i,j,k\}$. Let $k=4$ (the case $j=4$ is similar). Without loss of generality, consider 
	$$\varphi(A_{12})\varphi(A_{14}), \quad Y=\varphi^{-1}(\varphi(A_{12})\varphi(A_{14})).$$ We have $A_{14}=[A_{13}, A_{34}]$, and \eqref{derivation-rel} provides
	$\varphi(A_{12})\varphi(A_{14})=[\varphi(A_{12})\varphi(A_{13}),\varphi(A_{34})]$. The first part of the proof of this lemma implies 
	$Y\subseteq [Diag(A),A_{34}]\subseteq A_{34}$.
	Next, $[A_{13}, Y]=(0)$ and Lemma \ref{zero-commutator-implies-zero} imply $Y=(0)$. 
	
	It remains to consider the case $i=4$. Let 
	$Y=\varphi^{-1}(\varphi(A_{41})\varphi(A_{42}))$. We have: 
	\begin{equation}\label{condition 5}
		[Y, A_{31}]=[Y, A_{32}]=(0),
	\end{equation}
	\begin{equation}\label{condition 6}
		[Y, A_{4j}]=(0), \quad j\neq 4.
	\end{equation}
	Equalities \eqref{condition 5} and Lemma \ref{zero-commutator-implies-zero} imply $Y_{12}=Y_{21}=(0)$. On the other hand, by 
	Lemma~\ref{commutator-belongs-diagonal}, $[Y,A_{12}], [Y, A_{21}]\subseteq Diag(A)$. This implies: 
	\begin{equation}\label{condition 7}
		[Y, A_{12}]=[Y, A_{21}]=(0).
	\end{equation}
	Now, equalities \eqref{condition 5}-\eqref{condition 7}, together with Lemma \ref{zero-commutator-implies-zero}, imply that 
	$Y_{ij}=(0)$ for all $1\leq i\neq j\leq 4$, except for $i=3$ and $j=4$. Thus, we proved that $Y\subseteq Diag(A)+A_{34}$. Let $y\in Y,$
	$y=d+a_{34}$, where $d\in Diag(A)$, $a_{34}\in A_{34}$. From $[A_{43}, Y]=(0)$, it follows that $A_{43} a_{34}=(0)$, which implies 
	$Aa_{34}=(0)$, and therefore $e_3 a_{34}=a_{34}=0$. We proved that $Y\subseteq Diag(A)$. This completes the proof of the lemma.
\end{proof}
\begin{lem}\label{varphi-square-zero}
	For arbitrary indices $i\neq j$, we have $\varphi(A_{ij})\varphi(A_{ij})=(0)$.
\end{lem}
\begin{proof}
	By Lemma \ref{commutator-belongs-diagonal}, $\varphi(A_{ij})\varphi(A_{ij})\subseteq \varphi(Diag(A)\cap [A,A])$. Let $k\neq i,j$. Then:
	$$
	\begin{array}{l}
		\varphi(A_{ij})\varphi(A_{ij})=[\varphi(A_{ik}),\varphi(A_{kj})]\varphi(A_{ij})=[\varphi(A_{ik}),\varphi(A_{kj}) \varphi(A_{ij})]\subseteq\\
		\qquad \qquad \qquad \qquad \subseteq[\varphi(A_{ik}),\varphi(Diag(A)\cap [A,A])]\subseteq \varphi(A_{ik}),
	\end{array}
	$$
	where the first inclusion follows from Lemma \ref{preimage-in-diagonal}. Finally, 
	notice that $Diag(A)\cap A_{ik}=(0)$, which completes the proof of the lemma.
\end{proof}
\begin{lem}\label{three-different-zero}
	For arbitrary distinct indices $i,j,k$, we have $\varphi(A_{ij})\varphi(A_{ik})=\varphi(A_{ji})\varphi(A_{ki})=(0)$. 
\end{lem}
\begin{proof}
	Let us start with the case of Theorem \ref{main2}, where there are four idempotents. Let $1\leq t\leq 4$, $t\notin\{i,j,k\}$. Then $A_{ij}=[A_{it}, A_{tj}]$, and:
	$$
	\varphi(A_{ij})\varphi(A_{ik})=[\varphi(A_{it}),\varphi(A_{tj})] \varphi(A_{ik})=[\varphi(A_{it})\varphi(A_{ik}), \varphi(A_{tj})].
	$$
	Since $\varphi$ is a Lie homomorphism and $[A_{ik},A_{tj}]=(0)$, the left-hand side 
	lies in $\varphi(Diag(A)\cap[A,A])$, and the right-hand side lies in: 
	$$[\varphi(Diag(A)\cap[A,A]),\varphi(A_{tj})]\subseteq \varphi(A_{tj}).$$ Since  $Diag(A)\cap A_{tj}=(0)$, we conclude that 
	$\varphi(A_{ij})\varphi(A_{ik})=(0)$. 
	
	Next, consider the case of Theorem \ref{main}. Without loss of generality, we will focus on the product $\varphi(A_{12})\varphi(A_{13})$.
	Let $$Y=\varphi^{-1}(\varphi(A_{12})\varphi(A_{13})).$$ By \eqref{condition 3}, we have $[A_{12}, Y]=[A_{13}, Y]=(0)$. Following the proof of Lemma  \ref{preimage-in-diagonal},
	$[A_{23}, Y]=[A_{32},Y]=(0)$, and
	\begin{equation}\label{A1213}
		\varphi(A_{12})\varphi(A_{13})=[\varphi(A_{13}),\varphi(A_{32})] \varphi(A_{13})=[\varphi(A_{13}),\varphi(A_{32})\varphi(A_{13})].
	\end{equation}
	Let us show that $\varphi(A_{32})\varphi(A_{13})\subseteq [B,B]$. Indeed, $\varphi(A_{32})=[\varphi(A_{31}),\varphi(A_{12})]$.
	Hence, $$\varphi(A_{32})\varphi(A_{13})=[\varphi(A_{31}),\varphi(A_{12})]\varphi(A_{13})=[\varphi(A_{31})\varphi(A_{13}),\varphi(A_{12})].$$
	Then \eqref{A1213} implies that $Y\subseteq [A_{13},A]$. By Lemma \ref{preimage-in-diagonal}, $Y\subseteq Diag(A)$. Thus, we conclude that
	$Y\subseteq A_{11}+A_{33}$. Applying Lemma \ref{zero-commutator-implies-zero} to the equality $[Y,A_{12}]=(0)$, we get $Y\subseteq A_{33}$.  The equality $[Y,A_{32}]=(0)$  implies $Y=(0)$. 
	
	Similarly, $\varphi(A_{ji})\varphi(A_{ki})=(0)$. This completes the proof of the lemma.
\end{proof}
The following lemma makes sense only in the context of Theorem \ref{main2}.
\begin{lem} Let $1\leq i,j,k,t \leq 4$ be distinct. Then $\varphi(A_{ij})\varphi(A_{kt})=(0)$.
\end{lem}
\begin{proof} Without loss of generality, consider the product $\varphi(A_{12})\varphi(A_{34})$. Let $$Y=\varphi^{-1}(\varphi(A_{12})\varphi(A_{34})).$$ We have $A_{12}=[A_{13}, A_{32}]$. Hence, 
	$$
	\varphi(A_{12})\varphi(A_{34})=[\varphi(A_{13}),\varphi(A_{32})] \varphi(A_{34})=[\varphi(A_{13})\varphi(A_{34}),\varphi(A_{32})],
	$$
	which implies  $Y\subseteq [A, A_{32}]$. Similarly, since $A_{12}=[A_{14}, A_{42}]$, we have  $$\varphi(A_{12})\varphi(A_{34})=[\varphi(A_{14}),\varphi(A_{42})]\varphi(A_{34})=[\varphi(A_{14}),\varphi(A_{42})\varphi(A_{34})]$$  leading to $Y\subseteq [A_{14},A]$. Thus, $Y \subseteq [A,A_{32}]\cap[A_{14}, A]=A_{12}+A_{34}$.
	By Lemma \ref{three-different-zero}, $$[\varphi(A_{12})\varphi(A_{34}),\varphi(A_{24})]\subseteq \varphi(A_{14})\varphi(A_{34})=(0),$$ so $[Y, A_{24}]=(0)$. Lemma \ref{zero-commutator-implies-zero} then implies $Y\subseteq A_{34}$. Similarly, by Lemma \ref{three-different-zero},
	$$[\varphi(A_{13}),\varphi(A_{12})\varphi(A_{34})]\subseteq \varphi(A_{12})\varphi(A_{14})=(0),$$ so $[Y,A_{13}]=(0)$. Again, by Lemma \ref{zero-commutator-implies-zero}, it follows that $Y=(0)$. This completes the proof of the lemma and  Proposition \ref{speciality}.
\end{proof}

\section{Universal Specializations and Proofs of Theorem \ref{main} and  Theorem \ref{main2}}\label{section4}
\subsection*{4.1} Let $A$ be an associative ring with pairwise orthogonal idempotents $e_1$, $e_2$, $e_3$ (resp. $e_1$, $e_2$, $e_3$, $e_4$)
satisfying the assumptions of Theorem \ref{main} (resp. Theorem \ref{main2}). Then the ring $[A,A]$ is $\Delta$-graded,
where $\Delta$ is the root system $A_2$ (resp. $A_3$); see Lemma \ref{example-graded}. Recall that $A^{op}$ is the opposite ring of  $A$, $A^{op}=(A,a\cdot b=ba)$.
Consider the direct sum $A\oplus A^{op}$. Then the mapping $\theta: a\to a\oplus (-a^{op})$ defines a specialization of the $\Delta$-graded ring
$L=[A,A]$. 
\begin{lem}
	The associative ring $A\oplus A^{op}$ is generated by the subset $\cup_{\alpha\in\Delta}\,\theta(L_{\alpha})$.
\end{lem}
\begin{proof}
	For distinct indices $i,j,k$ and elements $a_{ij}\in A_{ij}$, $b_{jk}\in A_{jk}$, we have 
	$$
	(a_{ij}\oplus{ -a_{ij}^{op}})(b_{jk}\oplus{ -b_{jk}^{op}})=a_{ij} b_{jk}\oplus (b_{jk}a_{ij})^{op}=a_{ij} b_{jk}\oplus 0.
	$$
	Since $A_{ik}=A_{ij }A_{jk}$, it follows that $A_{ik}\oplus  (0)$ lies in the subring generated by $\cup_{\alpha\in\Delta} 
	\theta(L_{\alpha})$. 
	Since the ring $A$ is generated by $\cup_{i\neq j} A_{ij}$, it follows that
	$A\oplus (0)$  lies in this subring and similarly, so does $ (0) \oplus A^{op}$. This completes the
	proof of the lemma. 
\end{proof}
Let $\widehat{L}_{gr}\to L$ be the universal graded central extension of $L=[A,A]$,  and let $u:\widehat{L}_{gr}\to U$ be the universal specialization of the $\Delta$-graded Lie algebra $L$. The ring $U$ is
$\Gamma$-graded,  $U=U_0+\oplus_{\alpha\in \Delta} U_{\alpha}$, and is equipped with an involution $*: U\to U$.
For an arbitrary element $x\in u(L_{\alpha})$, where $\alpha\in\Delta$, we have $x^*=-x$. By the universality of the specialization 
$u$, there exists an epimorphism $\tau: U\to A\oplus A^{op}$ such that $\tau(u(a))=a\oplus -a^{op}$
for any element $a\in L_{\alpha}$, $\alpha\in \Delta$.  

In this section, we prove the following proposition.
\begin{prop}\label{annihilator_extension}
	The epimorphism $U\to A\oplus A^{op}$ is an annihilator extension.
\end{prop}

Since the mapping $\theta: L_{\alpha}\to (A\oplus A^{op})_{\alpha}$
is injective, it follows that the mapping $u:L_{\alpha}\to U_{\alpha}$ for $\alpha\in\Delta$, is also injective. 
Let $I=\operatorname{Ker}\tau$. It is clear that $I^*=I$ and $I\cap u(L_{\alpha})=(0)$
for $\alpha\in\Delta$.
\begin{lem}\label{product-of-two}
	For an arbitrary root $\alpha\in \Delta$, we have $$U_{\alpha}=\sum_{\mu,\nu\in\Delta,\, \mu+\nu=\alpha} u(L_{\mu})u(L_{\nu}).$$
\end{lem}
\begin{proof}
	We begin with two observations:
	\begin{enumerate}
		\item[(1)] if $\alpha$, $\beta$, $\gamma\in \Delta$ and $\alpha+\beta$, $\alpha+\gamma$, $\beta+\gamma\in\Delta$,
		then $\alpha=\omega_i-\omega_j$, $\beta=\omega_j-\omega_k$, $\gamma=\omega_k-\omega_i$ for some distinct indices $i,j,k$. Observe that $\alpha+\beta+\gamma=0$.
		\item[(2)] If $\alpha,\beta\in\Delta$, then either $\alpha+\beta\notin \Delta$ or $\alpha-\beta\notin \Delta$.
	\end{enumerate}
	An arbitrary element from $U_{\alpha}$ is a sum of products of the form $$u(a_{\mu_1})\cdots u(a_{\mu_r}), \quad \text{where} \quad  \mu_i\in\Delta, \quad 
	a_{\mu_i}\in L_{\mu_i}, \quad \text{and} \quad \sum \mu_i=\alpha.$$ We  refer to these products as regular products. Suppose that a product  $u(a_{\mu_1})\cdots u(a_{\mu_r})$ can not be represented as a sum of regular products of shorter length. We will show that $r\leq 2$. Let $\beta, \gamma\in\Delta$, $x_{\beta}\in L_{\beta}$, $y_{-\beta}\in L_{-\beta}$, and $z_{\gamma}\in L_{\gamma}$, assuming that $\gamma\neq \beta$. Then 
	$$
	u(x_{\beta}) u(y_{-\beta}) u(z_{\gamma})=u(x_{\beta}) u([y_{-\beta},z_{\gamma}])+u(x_{\beta})u(z_{\gamma}) u(y_{-\beta}).
	$$
	By  Observation $(2)$, we have $u(x_{\beta}) u(z_{\gamma}) u(y_{-\beta})=0$. If $\gamma=\beta$, then 
	\begin{equation}\label{tech1}
		u(x_{\beta}) u(y_{-\beta}) u(z_{\beta})=[u(x_{\beta}), u(y_{-\beta}) ]u(z_{\beta}).
	\end{equation} 
	There exist roots $\beta_1,\beta_2\in\Delta$ such that $\beta_1+\beta_2=\beta$ and $[L_{\beta_1}, L_{\beta_2}]=L_{\beta}$. Clearly, 
	$\beta_1,\beta_2\neq \pm \beta$. We have 
	$$
	[u(L_{\beta}), u(L_{-\beta})]\subseteq [u(L_{\beta_1}),u(L_{-\beta_1})]+[u(L_{\beta_2}),u(L_{-\beta_2})].
	$$
	Substituting this  into \eqref{tech1}, we obtain the case that has already been considered. Thus, we have shown that $u(x_{\beta})u(y_{-\beta})u(z_{\gamma})$ can be represented as a sum of regular products of length two. 
	Hence, we may assume that for any product $u(a_{\mu_1})\cdots u(a_{\mu_r})$, where $r\geq 3$, the condition $\mu_i+\mu_{i+1}\in \Delta$ holds 
	for $1\leq i\leq r-1$. Thus, we can commute $u(a_{\mu_i})$ and $u(a_{\mu_{i+1}})$ modulo a regular product 
	of length $r-1$.  In particular, for any distinct $1\leq i,j,k\leq r$ the roots $\mu_i,\mu_j,\mu_k$, satisfy the assumptions of Observation $(1)$. Then, for $3\leq p\leq r$, the condition $\mu_p=-\mu_1-\mu_2$  implies 
	$u(a_{\mu_1})\cdots u(a_{\mu_r})=0$. If $r=3$, then $\alpha=\mu_1+\mu_2+\mu_3=0,$ which leads to a contradiction.
	Thus, we conclude that $$U_{\alpha}= u(L_{\alpha})+\sum_{\mu,\nu \in\Delta,\, \mu+\nu=\alpha} u(L_{\mu}) u(L_{\nu}).$$ The first summand on the right-hand side is contained in the second summand. This completes the proof of the lemma.
\end{proof}
\begin{lem}\label{symmetric}
	An arbitrary element $x\in I_{\alpha}=I\cap U_{\alpha}$, where $\alpha\in\Delta$, is symmetric, i.e., $x^*=x$.
\end{lem}
\begin{proof}
	Let $x\in I_{\alpha}$. By Lemma~\ref{product-of-two},  we have $$x=\sum_{\mu_1,\mu_2 \in\Delta,\,\mu_1+\mu_2=\alpha} u(x_{\mu_1}) u(x_{\mu_2}), \quad \text{where} \quad  x_{\mu_i}\in L_{\mu_i} \quad \text{for} \quad  i=1,2.$$ Hence $$ x^*=\sum u(x_{\mu_2}) u(x_{\mu_1}), \ \  \text{and} \ \  x-x^*=\sum\, u([x_{\mu_1},x_{\mu_2}])\in  u(L_{\alpha}).$$ Since  $x-x^*\in I\cap u(L_{\alpha})=(0),$  it follows that $x^* =x.$ This completes the 
	proof of the lemma.
\end{proof}

\subsection*{4.2} Let us prove  Proposition~\ref{annihilator_extension} under the assumptions of Theorem~\ref{main}.
In fact, we will show that, in this case, the map $\tau: U\to A\oplus A^{op}$ is an isomorphism. Then, we will proceed to prove Theorem~\ref{main}. We have already mentioned, see \eqref{def-e}, that there exist elements $a_{12}^{(\mu)}\in A_{12}$, $b_{21}^{(\mu)}\in A_{21}$ such that 
$$\sum_{\mu} a_{12}^{(\mu)} b_{21}^{(\mu)}=e_1, \qquad  \text{and} \qquad \sum_{\mu} b_{21}^{(\mu)} a_{12}^{(\mu)}\in A_{22}.$$
Similarly, there exist elements $c_{32}^{(\nu)}\in A_{32}$, $d_{23}^{(\nu)}\in A_{23}$ such that 
$$\sum_{\nu} c_{32}^{(\nu)} d_{23}^{(\nu)}=e_3, \qquad  \text{and} \qquad \sum_{\nu} d_{23}^{(\nu)} c_{32}^{(\nu)}\in A_{22}.$$ Consider 
the elements: 
$$
E_1=\sum_{\mu} u(a_{12}^{(\mu)}) u( b_{21}^{(\mu)}) ,\quad E_3=\sum_{\nu} u(c_{32}^{(\nu)}) u( d_{23}^{(\nu)}).
$$
\begin{lem}\label{pierce-component}
	For arbitrary elements $x_{13}\in u(A_{13})$ and $y_{32}\in u(A_{32})$, we have $E_1 x_{13} y_{32}=x_{13} y_{32}$. 
\end{lem}
\begin{proof}
	Consider the commutator $\rho=\Sigma_{\mu} [u(a_{12}^{(\mu)}),u(b_{21}^{(\mu)})]$. Let $x_{13}=u(a_{13})$ for some $a_{13}\in A_{13}$. Then we have $$\sum_{\mu}\big[ [a_{12}^{\mu},b_{21}^{\mu}],a_{13} \big]=a_{13}, \quad  \text{which implies} \quad 
	[\rho, x_{13}] =u\big(\sum_{\mu} \big[ [a_{12}^{\mu},b_{21}^{\mu}],a_{13}\big]\big)=u(a_{13})=x_{13}.
	$$
	Thus, we obtain $\rho x_{13} y_{32}=x_{13} y_{32}+x_{13}\rho y_{32}$. Since $u$ is a specialization, it follows that 
	$$x_{13}\rho y_{32}=\sum_{\mu} x_{13}[u(a_{12}^{(\mu)}),u(b_{21}^{(\mu)})]y_{32}=0.$$ This completes the proof of the lemma.
\end{proof}

 Furthermore,
$$
x_{13}=\sum_{\nu} u\big([a_{13},[c_{32}^{(\nu)},d_{23}^{(\nu)}]]\big)=
\sum_{\nu} \big( x_{13} u(c_{32}^{(\nu)}) u(d_{23}^{(\nu)})+u(d_{23}^{(\nu)}) u(c_{32}^{(\nu)})x_{13}\big).
$$
Since we have $E_1 u(d_{23}^{(\nu)})=0$, it follows from Lemma~\ref{pierce-component} that
$$
E_1 x_{13}=\sum_{\nu} x_{13} u(c_{32}^{(\nu)}) u(d_{23}^{(\nu)})=x_{13} E_3.$$ On the
other hand,  
$$
\sum x_{13} u(c_{32}^{\nu}) u(d_{23}^{(\nu)})=x_{13}-\sum_{\nu} u(d_{23}^{\nu})u(c_{32}^{(\nu)})x_{13}=
x_{13}-E^*_3 x_{13}.
$$
We proved that 
\begin{equation}\label{E1E3}
	E_1x_{13}=x_{13} E_3=(1-E_3^*)x_{13}.
\end{equation}
\begin{lem}\label{x130} If $b=(1-E_1)x_{13}\in I$, where $x_{13}\in u(A_{13})$, then $x_{13}=0$.
\end{lem}
\begin{proof} Let $b=(1-E_1)x_{13}\in I_{\omega_1-\omega_3}$. Then, by Lemma~\ref{symmetric}, we have $b=b^*$. 
	By \eqref{E1E3}, 
	$E_1 x_{13}=(1-E_3^*) x_{13}$. Therefore, $b=(1-E_1)x_{13}=E_3^* x_{13}$ and $b^*=-x_{13}E_3$.
	Thus, $(1-E_1)x_{13}=-x_{13}E_3$,  or equivalently, $x_{13}-E_1 x_{13}=-x_{13} E_3$.  By \eqref{E1E3}, we have 
	$E_1 x_{13}=x_{13} E_3$, which implies  $x_{13}=0$.  This completes the proof of the lemma.
\end{proof}
\begin{lem}\label{E1-zero}
	For an arbitrary element $y\in I_{\omega_1-\omega_3}$, we have $E_1 y=y$.
\end{lem}
\begin{proof}
	There is one way to represent $\omega_1-\omega_3$ as a sum of two roots, namely
	$\omega_1-\omega_3=(\omega_1-\omega_2)+(\omega_2-\omega_3)$. Hence, by Lemma~\ref{product-of-two},
	$$I_{\omega_1-\omega_3}\subseteq u(A_{12})u(A_{23})+u(A_{23})u(A_{12})\subseteq u(A_{13})+u(A_{12})u(A_{23}).$$
	Let $y\in I_{\omega_1-\omega_3}$, and write $y=x_{13}+y'$, where $x_{13}\in u(A_{13})$ and $y'\in u(A_{12})u(A_{23})$. We have $$ u(A_{12})= u([A_{13},A_{32}])\subseteq u(A_{13}) u(A_{32})+u(A_{32})u(A_{13}).$$ Hence, $u(A_{12})u(A_{23})\subseteq u(A_{13}) u(A_{32}) u(A_{23}).$ By Lemma~\ref{pierce-component}, it follows that $E_1y'=y'$. Thus, $E_1 y=E_1 x_{13}+y'$. Therefore, $y-E_1 y=(1-E_1)x_{13}\in I$. By 
	Lemma~\ref{x130}, we obtain $x_{13}=0$, which proves the assertion of the lemma.
\end{proof}
\begin{lem} We have $I=I_0$ and $UI=IU=(0)$.
\end{lem}
\begin{proof}
	Analogously to the proof of Lemma~\ref{E1-zero}, we can show that for an arbitrary element $y\in I_{\omega_1-\omega_3}$, we have $y E_3=y$. Now, $I_{\omega_1-\omega_3}=E_1 I_{\omega_1-\omega_3}=I_{\omega_1-\omega_3} E_3$. Hence, $I_{\omega_1-\omega_3}=I_{\omega_1-\omega_3}^*=E_3^* I_{\omega_1-\omega_3}$. But 
	$$E_1 E_3^*=\sum_{\mu,\nu}u(a_{12}^{\mu})u(b_{21}^{\mu})u(d_{23}^{\nu})u(c_{32}^{\nu})=0.$$ This implies that 
	$I_{\omega_1-\omega_3}=(0)$ and similarly, $I_{\alpha}=(0)$ for an arbitrary root $\alpha\in\Delta$. Thus, we have proved that
	$I=I_0$. Since the ring $U$ is generated by the subset $\cup_{\alpha\in\Delta} U_{\alpha}$, it follows that $UI=IU=I$.
	This completes the proof of the lemma.
\end{proof}

\noindent \textit{Proof of Proposition \ref{annihilator_extension} under the assumptions of Theorem \ref{main}.} To complete the proof that $\tau: U\to A\oplus A^{op} $ is an isomorphism, we need to show that $I=(0)$. Define 
$$\rho_{\omega_1-\omega_3}=
\sum_{\mu}\big[u(a_{12}^{(\mu)}), u(b_{21}^{(\mu)})\big]+\sum_{\nu}\big[u(d_{23}^{(\nu)}), u(c_{32}^{(\nu)})\big].$$
From \eqref{E1E3}, it follows that
$(1-\rho_{\omega_1-\omega_3}) u(L_{\omega_1-\omega_3})=(0)$.
Similarly, for every root $\alpha\in\Delta$, there exists an element 
$\rho_{\alpha}\in\Sigma_{\gamma}[u(L_{\gamma}), u(L_{-\gamma})]$ such that $(1-\rho_{\alpha})u(L_{\alpha})=(0)$.
Let $\wp=\Pi_{\alpha\in\Delta}(1-\rho_{\alpha})$, where the factors appear in the arbitrary order. We claim that for an arbitrary root $\alpha\in\Delta$, $\wp u(L_{\alpha})=(0)$. Indeed, let $\wp=\wp'(1-\rho_{\alpha})\wp''$.
Clearly, $\wp''u(L_{\alpha})\subseteq u(L_{\alpha})\cdots$. Hence, 
$$\wp  u(L_{\alpha})\subseteq \wp'(1-\rho_{\alpha})u(L_{\alpha})\dots =(0).$$ This implies that
$\wp U=(0)$ and $\{x\in U\,|\, Ux=(0)\}=(0)$. Thus, $I=(0)$.
We have now proven  that under the assumption of Theorem~\ref{main},  $\tau: U\to A\oplus A^{op}$ is an isomorphism.

\begin{proof} [Proof of Theorem~\ref{main}] Let $B$ be a ring, and let $\varphi:[A,A]\to [B,B]$ be an isomorphism of
	Lie rings. By Proposition~\ref{speciality}, the homomorphism $\varphi: [A,A]\to B$ is a specialization. We have proved above that $\theta:[A,A]\to A\oplus A^{op},$ $\theta(a)=a \oplus (-a^{op}), $ $a\in[A,A],$ is the  universal specialization of 
	$[A,A]$. Hence, there exists a homomorphism $\chi: A\oplus A^{op}\to B$ such 
	that $\varphi(a)=\chi(a\oplus -a^{op})$, $a\in[A,A]$. Then, the compositions $\psi_1(a)=\chi(a\oplus 0)$ and $\psi_2(a)=\chi(0\oplus a^{op})$ define a homomorphism and an anti-homomorphism from $A$ to $B$, respectively,
	satisfying $\varphi(a)=\psi_1(a)-\psi_2(a)$. Hence, the Lie isomorphism $\varphi$ extends to a standard homomorphism. This completes the proof of 
	Theorem~\ref{main}. \end{proof}

\subsection*{4.3} We will now  prove Proposition~\ref{annihilator_extension} under the assumptions of Theorem~\ref{main2}, and then proceed to prove Theorem~\ref{main2}.

Let  $u:[A,A]\to U$ be the universal specialization of the $\Delta$-graded ring $[A,A].$ From what we have proven 
above, it follows that the subring $U'$ of $U$, generated by $u(A_{ij})$ for $1\leq i\neq j\leq 3$, is isomorphic to $A'\oplus (A')^{op}$, where $A'=(\Sigma_{i=1}^3 e_i)A(\Sigma_{i=1}^3 e_i)$. In particular, the subring $U'$ has an identity 
element $e\in U'_0$, $e=e^*$, $e=u(e_1 +e_2 + e_3)$.
\begin{lem}\label{lemma19}
	For an arbitrary root $\alpha\in\Delta$, we have $U_{\alpha}\subseteq (1-e)Ue+eU(1-e)+eUe$.
\end{lem}
\begin{proof}
	Without loss of generality, we   consider only two cases: $\alpha=\omega_1-\omega_2$ and $\alpha=\omega_1-\omega_4$. By Lemma~\ref{product-of-two}, $U_{\omega_1-\omega_2}\subseteq U'+u(A_{14})u(A_{42})$. For an arbitrary  
	element $x\in u(A_{14}) u(A_{42})$, we have $[u(h_{13}), x]=x$, which implies the assumption of the lemma for $U_{\omega_1-\omega_2}$. Again, by Lemma~\ref{product-of-two}, $$U_{\omega_1-\omega_4}\subseteq \sum_{i=2}^3 u(A_{1i})u(A_{i4})\subseteq eUe+eU(1-e)+(1-e)Ue,$$ since $u(A_{12})$, $u(A_{13})\subseteq eUe$. This completes the proof of the lemma.
\end{proof}
\begin{lem}
	For an arbitrary root $\alpha\in\Delta$, we have $I_{\alpha}=(0)$.
\end{lem}
\begin{proof}
	Suppose $0\neq x\in I_{\alpha}$. Then, by Lemma~\ref{lemma19}, $x=(1-e)xe+ex(1-e)+exe$. If $(1-e)xe\neq 0$,
	then $(1-e)xe\in I_{\alpha}$ and, therefore, $(1-e)xe=((1-e)xe)^*\in eU(1-e)$, a contradiction. Similarly, $ex(1-e)=0$, thus
	$x=exe$. We now identify an idempotent $e_i$, $1\leq i\leq 3$, with its image $e_i\oplus -e_i^{op}$ in $A\oplus A^{op}$
	and its image in $U'$. If $e_i xe_j\neq 0$, $i\neq j$, then  $e_i xe_j\neq (e_i xe_j)^*\in e_j U e_i$, a contradiction.
	Now, suppose $e_i x e_i\neq 0$. Let $1\leq i,j,k\leq 3$ be distinct. Since $e_i\in Ue_jU$, we have $e_i=\Sigma_{\mu} a_{\mu}e_j b_{\mu}$; $a_{\mu}, b_{\mu}\in u(L_{\omega_i-\omega_j})\cup u(L_{\omega_j-\omega_i})$. Similarly, 
	$e_j=\Sigma_{\nu} c_{\nu}e_k d_{\nu}$; $c_{\nu}, d_{\nu}\in u(L_{\omega_j-\omega_k})\cup u(L_{\omega_k-\omega_j})$.
	Thus, 
	$$e_i xe_i=\Sigma_{\mu,\nu} e_i x(e_i a_{\mu} e_j)(e_j c_{\nu} e_k)(e_k d_{\nu} e_j)(e_j b_{\mu} e_i).$$
	In each summand, $$e_i x(e_i a_{\mu} e_j)\in L_{\alpha +w_i - w_j}, \ e_i x(e_i a_{\mu} e_j)(e_j c_{\nu} e_k)\in L_{\alpha +w_i - w_j +w_j -w_k}=L_{\alpha +w_i - w_k}.$$  Hence, at least one of the elements $e_i x(e_i a_{\mu} e_j)$ or $e_i x(e_i a_{\mu} e_j)(e_j c_{\nu} e_k)$ belongs to $I_{\beta}$ for some $\beta\not= 0,$ and, therefore, must be equal to zero. This completes the proof of the lemma.
\end{proof}
Since the ring $U$ is generated by $\cup_{\alpha\in\Delta} U_{\alpha}$, if follows that $I\subseteq U_0$, $I\subseteq \operatorname{Ann}(U)$. Hence, $U\to A\oplus A^{op}$ is an annihilator extension. This completes the proof of Proposition~\ref{annihilator_extension}.
\begin{proof}[Proof of Theorem~\ref{main2}] Let $\varphi:[A,A]\to[B,B]$ be a Lie isomorphism. The Lie ring $[A,A]$ is $A_3$-graded and perfect; see Corollary \ref{Corollary_4+}. By Proposition~\ref{speciality},
	the homomorphism $\varphi:[A,A]\to B$ is a specialization. Let $u:[A,A]\to U$ be the universal specialization of the Lie ring 
	$[A,A]$. There exists a unique homomorphism $\psi: U\to B$ such that the diagram
	$$
	\xymatrix{[A,A]\ar[rr]^{u}\ar[rrd]_{\varphi}& &
		U\ar[d]^{\psi}
		\\ &&B}
	$$
	is commutative. Consider the specialization $\theta:[A,A]\to A\oplus A^{op}$. By the universality property, there exists a unique epimorphism 
	$\tau: U\to A\oplus A^{op}$ making the diagram 
	$$
	\xymatrix{[A,A]\ar[rr]^{u}\ar[rrd]_{\theta}& &
		U\ar[d]^{\tau}
		\\ &&A\oplus A^{op}}
	$$ 
	commutative. By Proposition~\ref{annihilator_extension}, $\tau$ is an annihilator extension. Let $v:\widehat{A\oplus A^{op}}\to A\oplus A^{op}$ be the universal annihilator extension of the ring $A\oplus A^{op}$.
	Then there exists a unique epimorphism $\chi_1: \widehat{A\oplus A^{op}}\to U$ making the diagram
	$$
	\xymatrix{\widehat{A\oplus A^{op}}\ar[rr]^{\chi_1}\ar[rrd]_{v}& &
		U\ar[d]^{\tau}
		\\ &&A\oplus A^{op}}
	$$ 
	commutative. Let $\zeta:\widehat{[A,A]}\to [A,A]$ be the universal central extension of the Lie ring $[A,A]$. We lift the homomorphism
	$\varphi$ to the homomorphism $\widehat{\varphi}: \widehat{[A,A]}\to B$, i.e., $\widehat{\varphi}=\varphi \circ \zeta.$  As  mentioned in the Introduction (see \eqref{gamma}), there
	exists a natural homomorphism
	$\gamma: \widehat{ [A\oplus A^{op},A\oplus A^{op}] } \to \widehat{A\oplus A^{op}}$ making the following diagram commutative:
	$$
	\xymatrix{\widehat{[A\oplus A^{op},A\oplus A^{op}]}\ar[rr]^{\qquad\gamma}\ar[rrd]& &
		\widehat{A\oplus A^{op}}\ar[d]
		\\ &&A\oplus A^{op}}
	$$ 
	 If $L_1\to L_2$ is an embedding of perfect Lie rings, then there exists a homomorphism 
	$\widehat{L_1}\to \widehat{L_2}$ making the  following diagram commutative:
	$$
	\xymatrix{\widehat{L_1}\ar[rr]\ar[d]& & \widehat{L_2}\ar[d]
		\\ L_1\ar[rr]&&L_2}
	$$ 
	 We fix the homomorphism $\xi:\widehat{[A,A]}\to\widehat{[A\oplus A^{op}, A\oplus A^{op}]}$, and define $\chi=\chi_1 \circ \psi$. Now, the 
	diagram
	$$
	\xymatrix{\widehat{[A,A]}\ar[r]^{\xi \qquad \quad}\ar[rd]_{\widehat{\varphi}} &\widehat{[A\oplus A^{op}, A\oplus A^{op}]}\ar[r]^{\qquad \gamma} &
		\widehat{A\oplus A^{op}}\ar[ld]_{\chi}\ar[r]^{\quad \chi_1}& U\ar[lld]^{\psi}
		\\ &B & &}
	$$ completes the proof of the existence of $\chi.$
    
Now, we prove the uniqueness of the homomorphism $\chi$. Suppose there exists   another homomorphism $\chi' :\widehat{A\oplus A^{op}}\to B$ that also makes the  following diagram commutative: 
$$
	\xymatrix{\widehat{[A,A]}\ar[rr]\ar[rrd]_{\widehat{\varphi}}& &
		\widehat{A\oplus A^{op}}\ar[d]^{\chi'}
		\\ &&B}
	$$ 
Let $\widehat{A_{ij}},$ for $i\not= j,$ be the preimage of $A_{ij}$ in $\widehat{[A,A]}.$ Let $\widehat{A_{ij}}'$ be the image of $\widehat{A_{ij}}$ in   $\widehat{A\oplus A^{op}}.$ By the commutativity of the diagrams, the homomorphisms $\chi$ and $\chi'$ coincide on  $\widehat{A_{ij}}'.$  Now, it remains to observe that the subspaces $\widehat{A_{ij}}',$ for $i\not= j,$ generate  $\widehat{A\oplus A^{op}}.$ This completes the proof of   Theorem~\ref{main2}.
\end{proof}
\begin{proof}[Proof of the Corollary~\ref{corollary-main2}]
	Suppose that $\varphi: [A,A]\to [B,B]$ is a Lie isomorphism and that $\operatorname{Ann}(\langle[B,B]\rangle)=(0)$.
	By Theorem~\ref{main2}, there exists a homomorphism $\chi:\widehat{A\oplus A^{op}}\to B$ making the following diagram commutative:
	$$
	\xymatrix{\widehat{[A,A]}\ar[rr]\ar[rrd]_{\widehat{\varphi}}& &
		\widehat{A\oplus A^{op}}\ar[d]^{\chi}
		\\ &&B}
	$$ 
	Clearly, $\chi(\operatorname{Ann}(\widehat{A\oplus A^{op}}))\subseteq \operatorname{Ann}(\langle[B,B]\rangle)=(0)$. Hence, $\chi$ is, in fact, a homomorphism $A\oplus A^{op}\to B$. Thus, $\varphi$ extends to a  standard homomorphism. This completes the proof of the corollary.
\end{proof}
\begin{Rem}\label{rem-ann} Suppose that a ring $A$ satisfies the assumptions of Theorem~\ref{main} or Theorem~\ref{main2},
	and in the latter case, that $\operatorname{Ann}(A)=(0)$. Then any automorphism of the Lie ring $[A,A]$
	extends to standard one. Indeed, we only need to show that $\operatorname{Ann}(A)=(0)$ implies $\operatorname{Ann}(\langle[A,A]\rangle)=(0)$. If $x\in A$ and $x[A,A]=(0)$, then $x A_{ij}=(0)$ for any $i\neq j$. Now, it remains to recall that the subset
	$\cup_{i\neq j} A_{ij}$ generates the ring $A$.
\end{Rem}
\begin{proof}[Proof of Theorem~\ref{main3}]
	Let $A$ be a ring satisfying the assumptions of Theorem~\ref{main} or Theorem~\ref{main2}, and in the 
	latter case, assume $\operatorname{Ann}(A)=(0)$. Let $d:[A,A]\to[A,A]$ be a derivation. Consider the ring $A'=A+A\varepsilon$,
	$\varepsilon^2=0$. This ring satisfies all the assumptions imposed on $A$. We have 
	$[A',A']=[A,A]+[A,A]\varepsilon$.
	For an element $a+b\varepsilon\in[A',A']$, where $a,b\in[A,A]$, define
	$f(a+b\varepsilon)=a+(d(a)+b)\varepsilon$. It is easy to see that $f:[A',A']\to [A',A']$ is a Lie automorphism. In view of the Remark~\ref{rem-ann}, the automorphism $f$ extends to a standard homomorphism $\psi_1-\psi_2$, where $\psi_1: A'\to A'$ is a homomorphism, and 
	$\psi_2: A'\to A'$ is an anti-homomorphism, with the property  $\psi_1(A')\psi_2(A')=\psi_2(A')\psi_1(A')=(0)$. 
	\begin{lem} Any element $a\in A$ can be represented as a sum of products 
		\begin{equation}\label{decomposition}
			a=\sum_{i} a_{i,1}\cdots a_{i,n_i}, 
		\end{equation}
		where $a_{ij}\in[A,A]$, all integers $n_i$ are even, and $$\sum_i a_{i,n_i}\dots a_{i,1}=0.$$
	\end{lem}
	\begin{proof}
		It is easy to see that the set $A^{even}$ of elements that can be represented in this way is a subring of $A$. Let $1\leq i,j,k\leq 4$ be distinct integers. Then, since $A_{ij}=A_{ik} A_{kj}$, an arbitrary element $a\in A_{ij}$ can be represented as a sum $a=\Sigma_{\mu} b_{\mu} c_{\mu}$, 
		$b_{\mu}\in A_{ik}$, $c_{\mu}\in A_{kj}$. Clearly, for each $\mu,$ we have $c_{\mu} b_{\mu}=0$. Thus, we have shown that $A_{ij}\subseteq A^{even}$. Since the set
		$\cup_{i\neq j} A_{ij}$ generates the ring $A$, the assertion of the lemma follows. \end{proof}
	Now, let us extend the Lie derivation $d:[A,A]\to[A,A]$ to a 
	derivation of the ring $A$. For an arbitrary element $a\in A,$ consider a presentation  of the form \eqref{decomposition}. Define
	\begin{equation}\label{tilded}
		\tilde{d}(a)=\sum_i(d(a_{i,1}) a_{i,2}\cdots a_{i, n_i}+a_{i,1}d(a_{i,2})a_{i,3}\cdots a_{i, n_i}+\cdots+a_{i,1}\cdots a_{i,n_i-1} d(a_{i,n_i})).
	\end{equation}
	Since a presentation of the form \eqref{decomposition} is not unique, we need to show that $$\sum_i a_{i,1}\cdots a_{i,n_i}=\sum_{i} a_{i,n_i}\cdots a_{i,1}=0,$$ when
	$a_{ij}\in[A,A]$ and all integers $n_i$ are even, implies that the expression \eqref{tilded} is equal to zero. This ensures that \eqref{tilded}  is well-defined. It is sufficient 
	to show that $\Sigma_i f(a_{i,1})\cdots f(a_{i, n_i})=0$. We have:
	$$\sum_i f(a_{i,1})\cdots f(a_{i,n_i})=\sum_i (\psi_1(a_{i,1})-\psi_2(a_{i,1}))\cdots(\psi_1(a_{i,n_i})-\psi_2(a_{i, n_i}))=$$ $$=\sum_i \psi_1(a_{i,1})\cdots\psi_1(a_{i,n_i})+\sum_i \psi_2(a_{i,1})\cdots\psi_2(a_{i,n_i})=$$ $$=\psi_1\big(\sum_i a_{i,1}\cdots a_{i,n_i}\big)+\psi_2\big(\sum_i a_{i,n_i}\cdots a_{i,1}\big)=0.
	$$
	Thus, the mapping $\tilde{d}$ is well-defined. From the definition, it immediately follows that $\tilde{d}$
	is a derivation of the ring $A$. Now, choose an element $a\in A_{ij}$, where $i\neq j$. Let $k\neq i,j$. Since $A_{ij}=[A_{ik}, A_{kj}]$, we can write 
	$$a=\sum_{\mu}[b_{\mu},c_{\mu}]=\sum_{\mu} b_{\mu} c_{\mu}-\sum_{\mu} c_{\mu} b_{\mu}, \quad \text{where} \quad b_{\mu}\in A_{ik}, \quad c_{\mu}\in A_{kj}.$$
	Thus, this is a presentation of the type \eqref{decomposition}. Hence, applying $\tilde{d},$ we obtain 
	$$
	\tilde{d}(a)=\sum_{\mu} d(b_{\mu}) c_{\mu}+\sum_{\mu} b_{\mu} d(c_{\mu})-\sum_{\mu} d(c_{\mu}) b_{\mu}-\sum_{\mu} c_{\mu} d(b_{mu})=
	d(a).
	$$
	Therefore, the restriction of the derivation $\tilde{d}$ to $[A,A]$ coincide with $d$. This completes the proof of Theorem~\ref{main3}.
\end{proof}

\begin{Rem}\label{rem-last} Theorems~\ref{main}, ~\ref{main2}, and ~\ref{main3} remain valid for algebras over arbitrary fields. To prove this, it suffices to replace  every occurrence of the word ``ring'' by the word ``algebra'' in the proofs of these theorems.
\end{Rem}

\section{An Example of a Non-Standard Lie Automorphism}\label{section5}

In this section, we construct an associative unital ring $A$ with full orthogonal idempotents $e_1,e_2\in A$, satisfying $e_1+e_2=1$, that extends to a non-standard automorphism of the Lie ring $[A,A]$. More precisely, we construct a derivation $d:[A,A]\to [A,A]$ that does not extend to a derivation 
of the ring $A$. 

As before, we consider the ring $A' =A+A \varepsilon, $ $\varepsilon^2 =0.$ Then, the mapping $\text{Id}+\varepsilon d$, where $\text{Id}$ is the identity mapping, define an  automorphism of the Lie ring $[A,A]$. In the proof of Theorem~\ref{main3}, it was shown that if $d$ does not extend to a derivation of the ring $A,$ then   $\text{Id}+\varepsilon d$ does not extend to a  standard homomorphism. 

Let $\mathbb{F}$ be a ground field with $\operatorname{char} \mathbb{F}\neq 2$. Let $B$ be an associative $\mathbb{F}$-algebra. A linear mapping $d: B\to B$ is called a {\it Jordan derivation} if it satisfies the identity 
$$d(ab+ba)=d(a)b+bd(a)+a d(b)+d(b)a \quad \text{for arbitrary elements} \quad  a,b\in B.$$ 
\begin{ex} Let $B$ be a unital associative $\mathbb{F}$-algebra given  by generators $e_i$, where $i\geq 1$, and $z$, subject to the relations $e_i e_j+e_j e_i=0$, $ze_i=e_i z$,
	$z^2=0$, $ze_i e_j e_k=0$ for arbitrary $i,j,k$.  It is easy to see that the elements $1$; $e_{i_1}\cdots e_{i_k}$, $1\leq i_1<\cdots< i_k$; $z$;
	$ze_i$, $i \geq 1$; $ze_i e_j$, $1 \leq i<j$, form a basis of the algebra $B$. Define the linear transformation $d: B\to B$ as follows: $d(1)=d(z)=0$, $d(e_i)=ze_i$,
	$d(e_{i_1}\cdots e_{i_k})=d(ze_i)=d(ze_i e_j)=0$, for $k\geq 2$. A straightforward verification shows that $d$ is a Jordan derivation, though it is not a derivation of the algebra $B.$
	
	Now, let $A=M_2(B)$ be the algebra of $2\times 2$ matrices over $B$. Then 
	$$
	[A,A]=\left\{x=\left(\begin{array}{cc} \Sigma_i x_i y_i& b\\ c& -\Sigma_i y_i x_i\end{array}\right) \,\Big|\, x_i,y_i,b,c\in B \right\}.
	$$
	Define the mapping
	$$
	\bar{d}(x)=\left(\begin{array}{cc} \Sigma_i d(x_i) y_i+\Sigma_i x_i d(y_i)& d(b)\\ d(c)& -\Sigma_i d(y_i) x_i-\Sigma_i y_i d(x_i)\end{array}\right).
	$$
	To show that the mapping $\bar{d}$ is well-defined, we need to verify that if $$\sum_i x_i y_i=\sum_i y_i x_i=0, \quad x_i, y_i\in B, \quad \text{then} \quad \sum_i d(x_i) y_i+\sum_i x_i d(y_i)=0.$$ Since $d(zB)=zd(B)=(0)$, without loss of generality, we assume that $x_i$ and $y_i$ are
	products of Grassmann generators $e_1, e_2,\dots$.  Let $l(x_i)$ and $l(y_i)$ denote the lengths of these products. Without loss of generality,
	 assume that $s=l(x_i)+l(y_i)$ does not depend on $i$. If $s\geq 3$, then $d(x_i) y_i=x_i d(y_i)=0$ for every $i$. Now, consider the case $s=2$. If $x_i=1$
	or $y_i=1$, then again $d(x_i)y_i=x_i d(y_i)=0$. It remains to consider the case when $x_i$ and $y_i$ are Grassmann variables. In this case,
	$$d(x_i)y_i+x_i d(y_i)=2zx_i y_i, \quad \sum_i d(x_i) y_i+\sum_i x_i d(y_i)=2z\sum_i x_i y_i=0.$$  Thus, we have shown that the mapping $\bar{d}$
	is well-defined. It is straightforward to verify that $\bar{d}$ is a derivation of the Lie algebra $[A,A]$. Let us now show that $\bar{d}$ does not extend 
	to a derivation of the algebra $A$. Consider the elements
	$$
	X=\left(\begin{array}{cc} e_1 & 0\\ 0& -e_1 \end{array}\right), \ Y=\left(\begin{array}{cc} 0 & e_2\\ 0& 0 \end{array}\right) .$$ These lie in  $ [A,A],$ and their product is $$ XY=\left(\begin{array}{cc} 0 & e_1e_2\\ 0& 0 \end{array}\right) \ \in \ [A,A].
	$$
	Now, applying $\bar{d}$, we obtain
     $$ \bar{d}(X)Y+X\bar{d}(Y)= \left(\begin{array}{cc} 0 & ze_1e_2\\ 0& 0 \end{array}\right) + \left(\begin{array}{cc} 0 & ze_1e_2\\ 0& 0 \end{array}\right)=\left(\begin{array}{cc} 0 & 2ze_1e_2\\ 0& 0 \end{array}\right)\neq 0.$$ However, $$   \bar{d} \left(\begin{array}{cc} 0 & e_1e_2\\ 0& 0 \end{array}\right)=0.$$  Thus, $\bar{d}$ does not extend to a derivation of the algebra $A.$ 
	Observe that the algebra $A=M_2(B)$ contains two orthogonal full idempotents: $$ \left(\begin{array}{cc} 1 & 0\\ 0& 0 \end{array}\right) \quad \text{and} \quad   \left(\begin{array}{cc} 0 &0\\ 0& 1 \end{array}\right).$$
\end{ex}
\begin{Rem} For a simpler example of a non-standard Lie derivation and a non-standard Lie automorphism in the matrix algebra $M_2 (\mathbb{F}),$ where $\operatorname{char} \mathbb{F}=2;$ see \cite{Bresar1}, Example 6.11.

\end{Rem}
\begin{Rem} Since Theorem~\ref{main3} follows from  Theorem~\ref{main}, the above example demonstrates  that Theorem~\ref{main} does not hold for rings containing only two orthogonal full idempotents. 
\end{Rem}

\section{Acknowledgments} 
The authors thank Matej Bre\v{s}ar for helpful discussions. 

The first author' work was partially supported by MES of Ukraine: Grant for the perspective development of the scientific direction “Mathematical sciences and natural sciences” at TSNUK. The second and third authors' work was partially supported by the NSFC Grant 12350710787.
The second author acknowledges financial support from Guangdong Basic and Applied
Basic Research Foundation grant 2024A1515013079.


\begin{thebibliography}{99}


\bibitem{BBCM} Beidar, K.I.,  Bre\v{s}ar, M., Chebotar, M.A., Martindale, W.S. 3rd.: On Herstein's Lie map conjectures I. Trans.
Amer. Math. Soc. \textbf{353}, 4235-4260 (2001)

\bibitem{BBCM2} Beidar, K.I., Bre\v{s}ar, M., Chebotar, M.A., Martindale, W.S. 3rd.: On Herstein's Lie map conjectures II. J.
Algebra \textbf{238}, 239-264 (2001)

\bibitem{BBCM3} Beidar, K.I., Bre\v{s}ar, M., Chebotar, M.A., Martindale, W.S. 3rd.: On Herstein's Lie map conjectures III. J.
Algebra \textbf{249}, 59-94 (2002)


\bibitem{BZ} Benkart, G., Zelmanov, E.: Lie algebras graded by finite root systems and intersection matrix algebras. Invent. Math. \textbf{126}, 1-45 (1996).


\bibitem{BM} Berman, S.,  Moody, R.V.: Lie algebras graded by finite root systems and the intersection matrix algebras of Slodowy. Invent. Math. \textbf{108}, 323-347 (1992)

\bibitem{Bezushchak2} Bezushchak, O.: Automorphism and derivations of algebras of infinite matrices. Linear Algebra App.  \textbf{650}, 42-59 (2022)


\bibitem{Bezushchak} Bezushchak, O.:  Derivations and automorphisms of locally matrix algebras. J. Algebra \textbf{576}, 1-26 (2021)


\bibitem{Bezushchak_Res_Math} Bezushchak, O.: Derivations of rings of infinite matrices.  Researches in Mathematics \textbf{31}(2), 14-18 (2023). 

\bibitem{BezOl_2}  Bezushchak, O.,  Oliynyk, B.: Primary decompositions of unital locally matrix algebras. Bull. Math. Sci. \textbf{10}(1), (2020). 


\bibitem{Bez_Ol__Sush} Bezushchak, O., Oliynyk, B., Sushchansky, V.: Representation of Steinitz's lattice in lattices of substructures of relational structures. Algebra Discrete Math. \textbf{21}(2), 184--201 (2016)

\bibitem{Bresar} Bre\v{s}ar, M.: Commuting traces of biadditive mappings, commutativity-preserving mappings and Lie mappings. Trans. Amer. Math. Soc. \textbf{335}, 525-546 (1993)



\bibitem{Bresar1} Bre\v{s}ar, M., Chebotar, M.A., Martindale, W.S. 3rd.: Functional Identities. Frontiers in Mathematics, Birkh\"auser Verlag, Basel-Boston-Berlin (2007)



\bibitem{Chebotar} Chebotar, M.A.:  On Lie automorphisms of simple rings of characteristic $2$. Fund. i Prikl. Matematika, 1257-1268 (1996)


\bibitem{Garland} Garland, H.: The arithmetic theory of loop groups. Inst. Hautes \'Etudes Sci. Publ. Math.  \textbf{52}, 5-136 (1980)


\bibitem{Her} Herstein, I.N.: Lie and Jordan structures in simple associative rings. Bull. Amer. Math. Soc. \textbf{67}, 517-531 (1961)

\bibitem{Kantor} Kantor, I.L.: Classification of irreducible transitively differential groups. Soviet Math. Dokl. \textbf{5}, 1404-1407 (1964)

\bibitem{Koecher} Koecher, M.:  Imbedding of Jordan algebras into Lie algebras I. Amer. J. Math. \textbf{89}, 787-816 (1967)


\bibitem{Mar} Martindale, W.S. 3rd.: Lie isomorphisms of prime rings. Trans.
Amer. Math. Soc. \textbf{142}, 437-455 (1969)


\bibitem{Schur} Schur, I.: \"Uber die Darstellung der endlichen Gruppen durch gebrochene lineare Substitutionen. Journal f\"ur die reine und angewandte Mathematik \textbf{127}, 20-50 (1904)

\bibitem{Tits} Tits, J.: Une classe d'alg\'ebres de Lie en relation avec les alg\'ebres de Jordan. Indag. Math. \textbf{24}, 530-534 (1962)


\bibitem{Z} Zelmanov, E.: Isomorphism of linear groups over an associative ring. Siberian Math. J. \textbf{26},  515-530 (1985)
\end{thebibliography}
\end{document}